\documentclass[psamsfonts, reqno]{amsart}

\usepackage{amscd,amsmath,amssymb,amsfonts,verbatim}
\usepackage{times}
\usepackage[cmtip, all]{xy}
\usepackage{graphicx}
\usepackage[active]{srcltx}
\usepackage{color}
\usepackage{textcomp}

\definecolor{refkey}{gray}{.5}   % graylevel for refs
\definecolor{labelkey}{gray}{.5} % graylevel for labels
\definecolor{Red}{rgb}{1,0,0}

\newcommand{\surj}{\twoheadrightarrow}	
		\newcommand{\wt}{\widetilde}

\newtheorem{theorem}{Theorem}[section]
\newtheorem{proposition}[theorem]{Proposition}
\newtheorem{lemma}[theorem]{Lemma}
\newtheorem{corollary}[theorem]{Corollary}
\newtheorem{question}[theorem]{Question}

\theoremstyle{definition}

\newtheorem{remark}[theorem]{Remark}
\newtheorem{example}[theorem]{Example}
\newtheorem{definition}[theorem]{Definition}

		\newcommand{\gs}{\sigma}

\newcommand{\BQ}{\mbox{$\mathbb Q$}}

\newcommand{\mm}{\mbox{$\mathfrak m$}}

\newcommand{\ot}{\mbox{\,$\otimes$\,}}	\newcommand{\op}{\mbox{$\oplus$}}

\def\proof{\paragraph{{\bf Proof}}}

\title{Monic inversion principle and complete intersection}
\author{Manoj K. Keshari and Soumi Tikader}

\begin{document}

\maketitle
\subjclass 2020 Mathematics Subject Classification:{13C10, 13B25}

 \keywords {Keywords:}~ {Projective modules, elementary orthogonal group, local global principle.}

 \begin{abstract}
Let $A$ be a regular ring of dimension $d$ essentially of finite type over an infinite field $k$ of characteristic $\neq 2$. Let $P$ be a projective $A$-module of rank $n$ with $2n\geq d+3$. Let $I$ be an ideal of $A[T]$ of height $n$ and 
$\phi:P[T]\surj I/I^2$ be a surjection. If $\phi\otimes A(T)$ has a surjective lift $\theta :P[T]\otimes A(T)\surj IA(T)$, then $\phi$ has a surjective lift $\Phi:P[T]\surj I$.
The case $P=A^n$ is due to Das-Tikader-Zinna \cite{DTZ}.
\end{abstract}

\section{Introduction}

{\bf Assumptions :} Rings are commutative noetherian with $1$ and projective modules are finitely generated of constant rank.

The monic inversion principle is a recurring theme in the area of projective modules and complete intersection ideals over polynomial  rings.
Let $A$ be a ring and $\wt P$ be a projective $A[T]$-module. Let $A(T)$ be the ring obtained by inverting all monic polynomials in $A[T]$. Horrocks \cite{H} proved that if $A$ is local, then $\wt P\ot A(T)$ is free if and only if $\wt P$ is free. Quillen \cite{Q} proved the local global principle that $\wt P$ is extended from $A$ if and only if $\wt P\ot A_{\mathfrak m}[T]$ is extended from $A_{\mathfrak m}$ for all $\mm \in Max(A)$.  Using local global principle and Horrocks result, Quillen \cite{Q} proved that $\wt P\ot A(T)$ is free if and only if $\wt P$ is free. 
Thus freeness property of projective $A[T]$-modules satisfies monic inversion principle. 

We say $\wt P$ has a unimodular element if there exist a surjection $\wt P\surj A[T]$ i.e. $\wt P\simeq Q\op A[T]$ for some projective $A[T]$-module $Q$. Roitman \cite{Roit} proved the analogue of Horrocks that if $A$ is local, then $\wt P\ot A(T)$ has a unimodular element if and only if $\wt P$ has a unimodular element. We do not yet have complete  analogue of Quillen though partial results are known. When $A$ contains $\BQ$, then Bhatwadekar-Sridharan \cite{BS1} proved that if $rank (\wt P) = dim(A)$, then 
$\wt P\ot A(T)$ has a unimodular element if and only if $\wt P$ has a unimodular element. Further if $A$ is a regular domain of dimension $d$ containing a field and $rank(\wt P)\geq \frac 12( d+3)$, then Bhatwadekar-Keshari \cite{BK} proved that $ \wt P\ot A(T)$ has a unimodular element if and only if $\wt P$ has a unimodular element. See \cite{KT} for recent results on this problem.

Now we will discuss monic inversion principle for complete intersection ideals.
Let $A$ be a regular domain of dimension $d$ containing a field $k$ and $P$ be a projective $A$-module of rank $n$ with $2n\geq d+3$. Let $I$ be an ideal of $A[T]$ of height $n$ and $\phi:P[T]\surj I/I^2T$ be a surjection. Bhatwadekar-Keshari \cite[Proposition 4.9]{BK} proved that $\phi\ot A(T)$ has a surjective lift $\theta: P\ot A(T) \surj IA(T)$ if and only if  $\phi$ has a surjective lift $\Phi:P[T]\surj I$.  If we further assume that $k$ is infinite perfect and
$A$ is essentially of finite type over $k$ i.e. $A$ is a localization of an affine $k$-algebra, then any surjection $\phi:P[T]\surj I/I^2T$ always lifts to a surjection $\Phi:P[T]\surj I$ \cite[Theorem 4.13]{BK}. This result answered a question of Nori and the proof uses above mentioned result \cite[Proposition 4.9]{BK}. We note that giving a surjection $\phi:P[T]\surj I/I^2T$ is equivalent to giving two surjections $\psi:P[T]\surj I/I^2$ and $f:P\surj I(0)$ such that $\psi(0) = f\ot A/I(0)$ \cite[Remark 3.9]{BS2}. If we are given a surjection $\phi:P[T]\surj I/I^2$ only, then $\phi$ need not have a surjective lift $\Phi:P[T]\surj I$, see the introduction of \cite{BK}. Thus we need some extra condition to ensure that the surjection $\phi :P[T]\surj I/I^2$ has a surjective lift $\Phi:P[T]\surj I$. One such condition was formulated by Nori that there exist a surjection $f:P\surj I(0)$ with $\phi(0)=f\ot A/I(0)$ which was considered in \cite{BK}. Another such condition was formulated in \cite{D,DTZ} as monic inversion principle.

\begin{question}\label{q}
Let $A$ be a regular domain of dimension $d$ containing a field $k$ and $P$ be a projective $A$-module of rank $n$ with $2n\geq d+3$. Let $I$ be an ideal of $A[T]$ of height $n$
and $\phi:P[T]\surj I/I^2$ be a surjection. Assume $\phi\ot A(T)$ has a surjective lift $\theta :P[T]\ot A(T)\surj IA(T)$. Does $\phi$ has a surjective lift $\Phi:P[T]\surj I$? 
\end{question}

Das-Tikader-Zinna \cite[Theorem 4.11]{DTZ} answered Question \ref{q} under additional assumption that $P=A^n$ is free, $A$ is essentially of finite type over $k$ and $k$ is infinite perfect of characteristic $\neq 2$. Their proof uses \cite[Theorem 4.13]{BK} crucially. We will strengthen this result in (\ref{4.13}) by removing perfect assumption,
extend the necessary results in \cite{DTZ} from free modules to projective modules and answer Question \ref{q} as follows (\ref{mt}).

\begin{theorem}
Let $A$ be a regular domain of dimension $d$ essentially of finite type over an infinite field $k$ of characteristic $\neq 2$. Let $P$ be a projective $A$-module of rank $n$ with $2n\geq d+3$. Let $I$ be an ideal of $A[T]$ of height $n$ and 
$\phi:P[T]\surj I/I^2$ be a surjection. If $\phi\otimes A(T)$ has a surjective lift $\theta :P[T]\otimes A(T)\surj IA(T)$, then $\phi$ has a surjective lift $\Phi:P[T]\surj I$.
\end{theorem}

A few words about the layout of the paper. Let $A$ be a ring with $1/2\in A$ and $P$ be a projective $A$-module of rank $n$. Let $\mathbb M(P)=P\op P^*\op A$ be the quadratic space with quadratic form $q(p,f,z)=f(p)+z^2$. We recall the definition of elementary orthogonal group $EO(\mathbb M(P))$ and prove that $EO(\mathbb M(A^n))=EO_{2n+1}(A)$. Using local global principle for $EO(\mathbb M(P[T]))$ due to Ambily-Rao \cite{AR} and a result of Stavrova \cite{St}, we prove that when $A$ is regular containing a field, then $O(\mathbb M(P[T]),T)=EO(\mathbb M(P[T]),T)$ and derive the splitting principle that if $b_1,b_2$ are comaximal, then any $\sigma \in EO(\mathbb M(P_{b_1b_2}))$ splits as $(\alpha_1)_{b_2}\circ (\alpha_2)_{b_1}$ where $\alpha_i \in EO(\mathbb M(P_{b_i}))$.
 
The main theorem is proved as follows.
Given a surjection $\phi:P[T]\surj I/I^2$, we get an element $(I,\phi)\in \mathcal {LO}(P[T])$ and hence an element $H(T)\in \mathbb Q'(P[T])$. Using a result of Mandal-Mishra \cite{MM}, we show that $\phi$ lifts to a surjection $\Phi:P[T]\surj I$ if and only if the elements $H(T)$ and $(0,0,1)$ of $\mathbb Q'(P[T])$ are connected by an element of $EO(\mathbb M(P[T]))$. Given that $\phi\ot A(T)$ has a surjective lift $\theta : P[T]\ot A(T)\surj IA(T)$, the elements $H(T)\ot A(T)$ and $(0,0,1)$ of $\mathbb Q'(P[T]\ot A(T))$ are connected by an element of $EO(\mathbb M(P[T]\ot A(T)))$. Using this we descend to the required.

The techniques and results of Mandal-Mishra \cite{MM} are crucially used.

\section{Elementary orthogonal group}

\begin{definition}\label{transvec} 
Let $A$ be a ring with $1/2\in A$.
The elementary orthogonal group  $EO_{2n+1}(A)$ is the subgroup 
of $O_{2n+1}(A)$ generated by the following five types of elementary orthogonal transvections defined as follows \cite[Definition 2.2]{DTZ}.
For $1\leq i\neq j\leq n$ and $\lambda\in A$, $(x_1,\cdots,x_n,y_1,\cdots,y_n,z) \in A^{2n+1}$
maps to
\begin{enumerate}
\item [(I)]
$(x_1,\cdots,x_{i-1},x_i-\lambda^2 y_i +2\lambda z,x_{i+1},\cdots,y_n,z-\lambda y_i).$ 
\item [(II)]
$(x_1,\cdots,y_{i-1},y_i-\lambda^2 x_i+2\lambda z,y_{i+1},\cdots,y_n,z-\lambda x_i).$
\item [(III)]
$ (x_1,\cdots,x_{i-1}, x_i+\lambda x_j,x_{i+1},\cdots, y_{j-1},y_j-\lambda y_i,y_{j+1},\cdots, y_n,z).$ 
\item [(IV)]
$ (x_1,\cdots,x_{i-1},x_i+\lambda y_j,\cdots, x_j-\lambda y_i, x_{j+1},\cdots,y_n,z).$ 
\item [(V)]
$  (x_1,\cdots, y_{i-1},y_i+\lambda x_j, \cdots,y_j-\lambda x_i,y_{j+1},\cdots, y_n,z).$

\end{enumerate}
 \end{definition}     
      
\begin{definition}\label{def} Let $A$ be a ring with $1/2\in A$ and $P$ be a projective $A$-module.  Let
   $\mathbb {M} (P) = \mathbb{H}(P) \perp A = P\oplus P^* \oplus A$ denote the projective module and also the quadratic space with quadratic form  
$q: \mathbb {M} (P) \rightarrow A$ given by $q (p,f,z)= f(p)+z^2$. Let $O(\mathbb M(P))$ be the group of orthogonal transformations of $\mathbb M(P)$. 

For linear maps $\alpha: A \rightarrow P$ and $\beta : A \rightarrow P^*$ with
dual maps $\alpha^*:P^*\to A$ and $\beta^*:P\to A$, define
$E_\alpha,E_\beta \in O(\mathbb{M}(P))$ 
by
\begin{enumerate}
\item $E_\alpha (p,f,z) = \left(p-\alpha\alpha^* (f)+2\alpha(z),f,z-\alpha^*(f)\right)$.

\item
$E_\beta (p,f,z) = \left(p,f-\beta\beta^*(p)+2\beta(z),z-\beta^*(p) \right)$.
\end{enumerate}

The elementary orthogonal group ${EO} (\mathbb{M}(P)) $ is the subgroup of $O(\mathbb{M}(P))$ generated by $E_\alpha$ and $E_\beta$ for all $\alpha,\beta$. Note $E_{-\alpha}=E_{\alpha}^{-1}$ and $E_{-\beta}=E_{\beta}^{-1}$. We will define some commutator relations where $[g,h]=g^{-1}h^{-1}gh$.

\begin{enumerate}
\item[(3)] 
If $ \beta^*(\alpha (1)) = 0$, then
$ E_{\alpha,\beta} = [E_{\beta}, E_{\alpha}] \in {EO}(\mathbb{M}(P))$ is defined as
$$E_{\alpha,\beta} (p,f,z) = (p+2\alpha\beta^*(p), f-2\beta\alpha^*(f),z )$$

\item[(4)]
If $\alpha_1 , \alpha_2 : A\rightarrow P $, then 
 $E_{\alpha_1, \alpha_2} = [E_{\alpha_2}, E_{\alpha_1}] \in {EO}(\mathbb{M}(P)) $ is defined as
$$E_{\alpha_1 ,\alpha_2} (p,f,z) =(p-2\alpha_2\alpha_1 ^*(f) +2 \alpha_1 \alpha_2^*(f) ,f,z)$$ 

\item[(5)]
If $\beta_1 , \beta_2 : A\rightarrow P^* $, then 
$E_{\beta_1 ,\beta_2} = [E_{\beta_2}, E_{\beta_1}] \in {EO}(\mathbb{M}(P)) $ is defined as 
 $$E_{\beta_1, \beta_2}(p,f,z) = ( p,f-2\beta_2\beta_1^*(p) +2 \beta_1\beta_2^*(p),z)$$
 
 \end{enumerate}

We note an identity which will be used in \ref{rmk6}. For $\alpha,\alpha':A\to P$ and $\beta,\beta':A\to P^*$,
$$ E_{\alpha+\alpha'} = E_{-\alpha/2,\alpha'}\circ E_{\alpha'}\circ E_{\alpha}~~\mbox{and}~~
E_{\beta+\beta'} = E_{-\beta/2,\beta'}\circ E_{\beta'}\circ E_{\beta}$$
\end{definition}

\begin{remark}
Roy \cite{Ro} defined elementary orthogonal group as the group generated by $E_\alpha$ and $E_\beta$ where $E_\alpha (p,f,z) = \left(p-\frac 12 \alpha\alpha^* (f)+\alpha(z),f,z-\alpha^*(f)\right)$ and
$E_\beta (p,f,z) = \left(p,f-\frac 12 \beta\beta^*(p)+\beta(z),z-\beta^*(p) \right)$.
We note that the two groups, the one defined by Roy and the one defined above,  are same. They both preserve the bilinear form $B$ defined by $B((p,f),(p',f'))=q(p+p',f+f')-q(p,f)-q(p',f')$, where $q$ is the quadratic form of $\mathbb H(P)$ defined by $q(p,f)=f(p)$. The difference in the definition of generators is due to the fact that Roy uses bilinear form $B$ to define the quadratic form whereas we use $\frac 12 B$.
\end{remark}

\begin{lemma}\label{vevpt}
Let $A$ be a ring with $1/2\in A$ and $P$ be a projective $A$-module.
If $\phi \in {EO}(\mathbb{M}(P))$, then there exist $\Phi(T) \in {EO}(\mathbb{M}(P[T]))$ such that $\Phi(0) = Id_{\mathbb{M}(P)}$ and $\Phi(1)= \phi$ i.e. every element of $  {EO}(\mathbb{M}(P))$ is homotopic to identity.
 \end{lemma}

\proof
It is enough to assume that $\phi$ is one of the two types of generators of ${EO}(\mathbb{M}(P)).$
If $\phi = E_\alpha$ for $\alpha : A\rightarrow P$, then $\Phi(T)=E_{{\alpha(T)}}$, where
${\alpha(T)}: A[T] \rightarrow P[T]$ is defined by ${\alpha(T)} (1)= T\alpha (1).$  
Similarly, if $\phi = E_\beta$ for $\beta : A\rightarrow P^*$, then $\Phi(T)=E_{{\beta(T)}}$, where 
${\beta(T)}: A[T] \rightarrow P^*[T]$ is defined by ${\beta(T)} (1)= T\beta (1).$
\qed

\begin{lemma} \label{rmk6}
Let $A$ be a ring with $1/2\in A$. Then $ EO_{2n+1}(A) %\subseteq 
={EO}(\mathbb{M}(A^n))$.
\end{lemma}

\begin{proof}

For forward inclusion we will show that five type of generators of $EO_{2n+1}(A)$ is a generator of ${EO}(\mathbb{M}(A^n)).$  Let $\delta_{ij}$ be the Kronecker delta function. 

\begin{enumerate}
\item[1]
For (I),  take $E_\alpha$, where $\alpha : A \rightarrow A^n $ with $\alpha(1)=\lambda  e_i$ and $\alpha^* : A^n \rightarrow A$ with $ \alpha^*(e_k)= \lambda \delta_{ik}$. 

\item[2]
For (II), take $E_\beta$, where $\beta : A \rightarrow A^n $ with $\beta(1) =\lambda  e_i$ and $\beta^* : A^n \rightarrow A$ with $ \beta^*(e_k)= \lambda \delta_{ik}$.
 
\item[3]
For (III), take  $E_{\alpha,\beta}$, where  $\alpha : A \rightarrow A^n $ with $\alpha(1)= e_i$ and $\beta^* : A^n \rightarrow A$ with $ \beta^*(e_k)= 
\frac 12 \lambda \delta_{jk}$.

\item[4] 
 For (IV), take  $E_{\alpha_1, \alpha_2}$, where $\alpha_1,\alpha_2: A \rightarrow A^n$ with $\alpha_1(1)=e_i$ and $\alpha_2 (1)= \frac 12 \lambda e_j.$
 
\item[5]
For type (V), take $E_{\beta_1, \beta_2}$, where $\beta_1,\beta_2: A \rightarrow A^n$ with $\beta_1(1)=e_i$ and $\beta_2 (1)=\frac 12 \lambda e_j.$
\end{enumerate}

For reverse inclusion we need to show that given $\alpha:A\to A^n$ and $\beta : A\to (A^n)^*$, $E_\alpha,E_\beta\in EO_{2n+1}(A)$. We will prove for $E_\alpha$ as $E_\beta$ case is similar. 

If $\alpha(1) = a_ie_i$, then $E_\alpha$ is of type (I). Let $\alpha(1)=\sum a_i e_i$. Define $\alpha_i:A\to A^n$ by $\alpha_i(1)=a_i e_i$. Then $\alpha=\sum_1^n \alpha_i = \alpha_1+\alpha''$, where $\alpha''=\sum_2^n \alpha_i$. By induction on $n$,
$E_{\alpha''} \in EO_{2n+1}(A)$.
We have observed earlier that 
$E_{\alpha_1 + \alpha''} = E_{-\alpha_1/2,\alpha''
} \circ E_{\alpha''}\circ E_{\alpha_1}$, where $E_{-\alpha_1/2,\alpha''
}=[E_{\alpha''},E_{-\alpha/2}]$ is the commutator. Thus $E_\alpha\in  EO_{2n+1}(A)$. $\hfill \square$
\end{proof}
\medskip

The following local global principle follows from Ambily-Rao \cite[Theorem 3.10]{AR}.

\begin{theorem}\label{th2} 
Let $A$ be a ring with $1/2\in A$ and $P$ be a projective $A$-module of rank $\geq 2$.
Let $\sigma(T) \in O(\mathbb M(P[T]),T)$ such that $\sigma_{\mathfrak{m}}(T)  \in {EO}(\mathbb{M}(P_{\mathfrak{m}}[T]))$ for all $\mathfrak{m}\in Max(A)$.  Then $\sigma(T) \in {EO}(\mathbb{M}(P[T]), T).$ 
\end{theorem}

As a consequence of (\ref{th2}), we have the following.

\begin{corollary}\label{th3}
Let $A$ be a ring with $1/2\in A$ and $P$ be a projective $A$-module of rank $\geq 2$.
Let $\sigma(T) \in O(\mathbb{M} (P[T]))$ such that $\sigma(0)\in {EO}(\mathbb{M}(P))$.
\begin{enumerate} 
\item If $\sigma_{\mathfrak{m}}(T)  \in {EO}(\mathbb{M}(P_{\mathfrak{m}}[T]))$ for all $\mathfrak{m}\in Max(A)$,  then $\sigma(T) \in {EO}(\mathbb{M} (P[T])).$ 

\item  Assume $\sigma(0)=Id$ and $a,b \in A$ be comaximal. If $\sigma_a (T) \in {EO}(\mathbb{M}(P_a[T]),T)$ and $\sigma_b (T) \in {EO}(\mathbb{M}(P_b[T]),T) $, then $\sigma(T) \in {EO}(\mathbb{M}(P[T]),T).$
\end{enumerate}
\end{corollary}

\begin{theorem}\label{th1}
Let $A$ be a regular ring containing a field of characteristics $\neq 2$ and $P$ be a projective $A$-module of rank $n\geq 2$. Then 

\begin{enumerate}\label{1l}
\item $O(\mathbb{M}(P[T]),T) = {EO}(\mathbb{M}(P[T]),T)$.
\item Let $\sigma(T) \in O(\mathbb{M}(P[T])).$ Then $\sigma(T) \in {EO}(\mathbb{M}(P[T]))$ if and only if $\sigma(0) \in {EO}(\mathbb{M}(P)).$ 
\end{enumerate}
\end{theorem}

\proof
(1) Let $\sigma(T)\in O(\mathbb{M}(P[T]),T) $. To show $\sigma(T)\in EO(\mathbb{M}(P[T]),T) $,  using local-global principle (\ref{th2}),
we may assume $A$ is a local ring
  and hence $P= A^n$ is free. By Stavrova \cite[Theorem 1.3]{St}, $O(\mathbb {M}(A[T]^{n}),T)=O_{2n+1}(A[T],T) = EO_{2n+1}(A[T],T)$ 
and by (\ref{rmk6}),  $  EO_{2n+1}(A[T]) = {EO}(\mathbb{M}(A[T]^n ))$. Thus $\sigma(T) \in {EO}(\mathbb{M}(A[T]^n ),T)$ and we are done.

(2) If $\sigma(T) \in O(\mathbb{M}(P[T]))$ with $\sigma(0) \in {EO}(\mathbb{M}(P))$, then
$\sigma(T)\circ \sigma(0)^{-1}\in O(\mathbb{M}(P[T]),T) = EO(\mathbb{M}(P[T]),T)$ by (1). Thus $\sigma(T)\in EO(\mathbb{M}(P[T]))$.
\qed

\begin{lemma} \label{th4}
Let $A$ be  a regular ring containing a field of characteristic $\neq 2$. Let $a \in A$ be a non-zerodivisor and $\sigma (T) \in EO( \mathbb{M}  (P_a[T]),T)$.  Then for all $n\gg 0$, given $ c-d \in a^nA$,  there exist $\tau(T)\in EO(\mathbb M(P[T]),T)$ such that $\tau(T)_a = \sigma(cT) \circ \sigma(dT)^{-1}$.
\end{lemma}

\proof
Let $R= End (P \oplus P^* \oplus A)$ be non-commutative ring. Since $\sigma (T) \in (Id +T R_a [T])^ {*},$ by Quillen \cite[Lemma 1]{Q}, there exist $\tau (T) \in (Id +T R[T])^ {*} $ such that  $\tau(T)_a =  \sigma(cT) \circ \sigma(dT)^{-1}.$  Since $\tau(T)_a \in EO(\mathbb M(P_a[T]))$, it preserves the quadratic form, therefore $\tau (T) \in O(\mathbb(P[T]),T) = EO( \mathbb{M} P[T],T)$, by (\ref{1l}).   
\qed

\begin{corollary}\label{split1}
 Let $A$ be a regular ring containing a field of characteristic $\neq {2}$ and $P$ be a projective $A$-module of rank $\geq 2$. Let $a,b \in A$ be comaximal and $\sigma (T) \in EO( \mathbb{M}( P_{ab}[T]),T)$. Then there exist $\alpha(T) \in EO( \mathbb{M} (P_a[T]),T)$ and $\beta (T) \in EO( \mathbb{M} (P_b[T]),T)$ such that $\sigma (T) = \alpha (T)_b \circ\beta (T)_a. $  
\end{corollary}

 \proof
 By (\ref{th4}), for all $ n \gg 0$, given $c - d \in a^{n} A $, there exist $\tau (T) \in EO(\mathbb{M} (P_b[T]),T)$ such that $\tau(T)_a = \sigma(cT)\circ\sigma(dT)^{-1}$.
Similarly given $c - d \in b^{n} A $, then there exist $\theta (T) \in EO(\mathbb{M} (P_a[T]),T)$ such that $\theta(T)_b = \sigma(cT)\circ\sigma(dT)^{-1}$.

If $a^ns+b^nr = 1$ for  $s,r \in A,$ then $\sigma(T) = (\sigma(T)\circ \sigma (a^ns T)^{-1})\circ( \sigma (a^ns T) \sigma (0)^{-1})=\alpha(T)_b\circ \beta(T)_a$ with $\alpha(T),\beta(T)$ as required. 
 \qed

\begin{corollary}\label{split}
Let $A$ be a regular ring containing a field of characteristic $\neq 2$
and $P$ be projective $A$-module of rank $\geq 2$. Let $a, b \in A$ be comaximal and $\sigma \in {EO}(\mathbb{M}(P_ {ab}))$. Then there exist  
$\alpha \in {EO}(\mathbb{M}(P_a ))$ and $\beta \in  {EO}(\mathbb{M}(P_b))$ such that $\sigma = \alpha_b  \circ \beta_a $.
\end{corollary}

\proof
By (\ref{vevpt}), there exist $\tau(T) \in {EO (\mathbb M(P_{ab}[T]),T )}$ with $\tau(1) = \sigma.$  By (\ref{split1}), $\tau(T) = \alpha(T)_b \circ \beta(T)_a$ with $\alpha(T) \in {EO}(\mathbb{M} (P_a[T]),T)$ and $\beta(T) \in  {EO}(\mathbb{M} (P_b[T]),T).$ Thus $\sigma=\tau(1) = \alpha(1)_b \circ  \beta(1)_a$ with $\alpha(1) \in {EO}(\mathbb{M}(P_a))$ and $\beta(1) \in {EO}(\mathbb{M}(P_b)).$ 
\qed

\section{Removing perfect hypothesis on field from \cite[Theorem 4.13]{BK}}

\begin{theorem}\label{4.13}
Let $A$ be a regular domain of dimension $d$ essentially of finite type over an infinite field $k$. Let $P$ be a projective $A$-module of rank $n$ with $2n\geq d+3$. Let $I$ be an ideal of $A[T]$ of height $n$ and $\phi:P[T]\surj I/I^2T$ be a surjection. Then $\phi$ has a surjective lift $\Phi:P[T]\surj I$. 
\end{theorem}

\proof
When $A$ is a $k$-spot, then $P=A^n$ and the result is proved in \cite[Theorem 4.2]{D1}. In the general case, let $\Sigma$ be the set of all $s\in A$ such that $\phi_s$ lifts to a surjection $\Psi:P_s[T]\surj I_s$. Then it is implicitely proved in \cite[Theorem 4.13]{BK} that $\Sigma$ is an ideal. Therefore $\Sigma=A$ as the result is true for $k$-spots and the result follows.
\qed

%%%%%%%%%%%%%%%%%%%%%%%%%%%%%%%%%
\section{Monic inversion principle for quadratic form}

We recall some definitions from Mandal-Mishra \cite{MM}.
\begin{definition}
 Let $A$ be a  ring with $1/2\in A$ and $P$ be a projective $A$-module. 

\begin{itemize}
\item
$\mathbb{Q}  (P)= \{ (p,f,s) \in \mathbb M(P)| s(1-s) + f(p) = 0 \}$. 
\item
$\mathbb{Q'}  (P)= \{ (p,f,z) \in \mathbb(P) | z^2 + f(p) = 1 \}$.
We write $\mathbb{Q'}  (A^n)$ as $\mathbb{Q'}_{2n+1}(A) = \{(x_1,\ldots,x_n,y_1,\ldots,y_n,z) \in \mathbb M(A^n)| \sum x_iy_i+z^2 = 1 \}.$

\item
   $ O(\mathbb{M}(P))=$  the group of isometries of the quadratic form $\mathbb M(P)$.
    We write $O(\mathbb{M}(A^n))$ as $O_{2n+1}(A).$

  \item
  $O(\mathbb{M}(P[T]),T)= \{\sigma(T) \in O(\mathbb{M}(P[T])) :  \sigma(0) = id \}$.

\item A local $P$-orientation is a pair $(I,\omega)$, where $I$ is an ideal of $A$ and $\omega:P\surj I/I^2$ is a surjection. The surjection $\omega$ is identified with the induced surjection $\omega:P/IP \surj I/I^2$.

\item $\mathcal{LO}(P) =$  the set of all local $P$-orientations.

\item 
Given a ring homomorphism $\gamma : A \rightarrow B$, the functor $-\otimes_A B$ gives natural maps $\mathcal{LO}  (P)  \rightarrow \mathcal{LO}  (P \otimes B)$,
$ \mathbb Q(P)\to \mathbb{Q}  (P\otimes B)$ and $\BQ'(P) \to \mathbb{Q'}  (P\otimes B)$.

\item Let $F(P)$ be one of $\mathcal{LO}  (P) $, $\mathbb{Q}  (P)$ or $\mathbb{Q'}  (P).$  
The homotopy orbit  $\pi_0(F(P))$ is defined by the pushout diagram in sets.
$$
\xymatrix{
         & F (P[T])\ar^{T=0} [r] \ar^{T=1} [d]  & F(P) \ar [d]
         &
           &   \\
        & F(P) \ar [r] & \pi_0(F(P) ) &
          & 
        }$$

\item 
Let $\Phi_0=(p,f,s)$ and $\Phi_1=(p' ,f',s')$ be elements of $\mathbb Q(P)$. Then $ [\Phi_0] = [\Phi_1]\in \pi_0(\mathbb{Q}  (P)) $  if and only if  there exist $\Phi(T) \in \mathbb{Q}  (P[T])$  such that  $\Phi(0)= \Phi_0$ and $\Phi(1) =\Phi_1$.  The element
$[(0,0,0)] \in \mathbb Q(P)$ is the base point. 

\item 
Let $\Phi_0=(p,f,s)$ and $\Phi_1=(p' ,f',s')$ be elements of $\mathbb Q'(P)$. Then $ [\Phi_0] = [\Phi_1]\in \pi_0(\mathbb Q'(P))$  if and only if  there exist $\Phi(T) \in \mathbb Q'(P[T])$ such that  $\Phi(0)= \Phi_0$ and $\Phi(1) =\Phi_1$.  The element
$[(0,0,1]\in \mathbb Q'(P)$ is the base point.

\item For local $P$-orientations $(I,\omega),(I',\omega')\in \mathcal{LO}(P)$, $[(I,\omega)]=[(I',\omega')] \in \pi_0(\mathcal{LO}  (P)) $ if and only if  there exist an ideal $K\subset A[T]$ and a local $P[T]$-orientation $(K,\sigma (T)) \in \mathcal{LO}  (P[T]) $ such that  $(K(0),\sigma(0)) = (I, \omega)$ and $(K(1),\sigma(1)) = (I',\omega')$.
   
\end{itemize}
\end{definition}

We state a result from Mandal-Mishra \cite[Lemma 2.4]{MM}.

\begin{lemma}\label{4}
 Let $A$ be ring with $1/2\in A$ and $P$ be a projective $A$-module. Then there is a base point preserving bijection $\pi_0 ({\mathbb{Q} } (P)) \simeq \pi_0 ({\mathbb{Q'} } (P))$ defined by $(p,f,s) \mapsto (2p,2f,1-2s)$. 
\end{lemma}

\begin{lemma} {\label{extended1}}
Let $A$ be a ring with $1/2\in A$ and  $P$ be a projective $A$-module of rank $\geq 2$. Then the natural map
$\mu : \pi_0(\mathbb{Q'} (P)) \simeq \pi_0(\mathbb{Q'} (P[T]))$ defined by $[v]\mapsto [v]$ is a bijection.
\end{lemma}

\proof
To see that $\mu$ is well defined, let $v,v'\in \mathbb{Q'}(P)$
with $[v]=[v']$ in $\pi_0(\mathbb Q'(P))$. Then there exist $H(W)\in \mathbb{Q'}(P[W])$ such that 
$H(0)=v$ and $H(1)=v'$. Then $H(W)\in \mathbb{Q'}(P[T,W])$ gives
that $[v]=[v']$ in $\pi_0(\mathbb{Q'}(P[T]))$.

Let $v,v'\in \mathbb{Q'}(P)$ be such that $[v]=[v']$ in $\pi_0(\mathbb{Q'}(P[T]))$.
Then there exist $H(T,W)\in \mathbb{Q'}(P[T,W])$ such that $H(T,0)=v$ and $H(T,1)=v'$. Then $G(W)=H(0,W)\in \mathbb{ Q'}(P[W])$ and
also $G(0)=v$ and $G(1)=v'$. Thus $\mu$ is injective.

Let $H(T)\in \mathbb{Q'}(P[T])$. Then $G(T,W)=H(TW)\in \mathbb{Q'}(P[T,W])$ with $G(T,0)=H(0)$ and $G(T,1)=H(T)$. Thus $[H(T)]=[H(0)]$ in $\pi_0(\mathbb Q'(P[T]))$. Thus $\mu$ is surjective.
\qed
\medskip

Using (\ref{th1}), the next result is a restatement of Mandal-Mishra \cite[Theorem 3.1, 3.4]{MM}.

\begin{proposition} \label{extended}
Let $A$ be a regular ring containing a field of characteristic $\neq 2$ and  $P$ be a projective $A$-module of rank $\geq 2.$ 
\begin{enumerate}
\item If $H(T) \in {\mathbb{Q'} } (P[T]) $, then there exist $\sigma (T) \in  {EO}(\mathbb{M} (P[T]),T)$ such that $ H(T) = \sigma(T) (H(0))$.
\item Consider the action of $EO(\mathbb M(P[T]),T)$ on $\mathbb Q'(P)$ given by $\sigma(T) \cdot v=\sigma(1)(v)$. Then the natural map $ {\mathbb{Q'} } (P)/ {EO}(\mathbb{M}(P[T]),T)  \to \pi_0 ({\mathbb{Q'} } (P)) $ is a bijection.
\end{enumerate}
\end{proposition}
 
In case $P=A^n$ is free, the following monic inversion is proved in Das-Tikader-Zinna \cite[Theorem 5.8]{DTZ}. We will closely follow their proof. For $v,v'\in \mathbb Q'(P)$,$v\equiv v'$ mod $EO(\mathbb M(P))
$ means $v'= v\sigma$ for some $\sigma\in EO(\mathbb M(P))$.

\begin{theorem} \label{monicquadratic}
Let $A$ be a regular domain containing a field $k$ of characteristic $\not = 2$ and $P$ be a projective $A$ module of rank $\geq 2$. Let $H(T) \in \mathbb{Q'}(P[T])$ with
$H(T)_g \equiv  (0,0,1)$ mod $EO(\mathbb M(P[T]_g))$
for some monic polynomial $g\in A[T]$. Then
$H(T) \equiv (0,0,1)$ mod $EO(\mathbb M(P[T]))$.
\end{theorem}

\proof
By \ref{extended}(1),  
 $H(T) \equiv H(0)$ mod $EO(\mathbb M(P[T]),T)$. Since  $H(0)=H(1)$ in $\pi_0(\mathbb Q'(P))$, by \ref{extended}(2),
$H(0)\equiv H(1)$ mod $EO(\mathbb{M}(P))$. Thus  $H(T) \equiv H(1)$ mod $EO(\mathbb M(P[T]))$.
It is given that $(0,0,1)\equiv H(T)_g\equiv H(0)\equiv H(1)$ mod ${EO}(\mathbb{M}(P[T]_g))$. 

In case $g=T$, putting $T=1$  gives $[(0,0,1)] \equiv H(1)$ mod ${EO}(\mathbb{M}(P))$. Thus $H(T) \equiv H(1)\equiv [(0,0,1)]$ mod $EO(\mathbb M(P[T]))$. 

For general $g$,   
since $H(T)_g \equiv H(1) \equiv (0,0,1)$ mod ${EO}(\mathbb{M}(P[T]_g))$,  there exist $\sigma\in {EO}(\mathbb{M}(P[T]_g))$ with
 $H(1)\sigma  = (0,0,1)$.

Let 
$g^{\ast}=T^{-\text{ deg }g}g\in A[T^{-1}]$. Then $g^{\ast}(T^{-1}=0)=1$ and 
$A[T^{-1},T]_{g^{\ast}}=A[T,T^{-1}]_{g}$.
Applying (\ref{split}) for comaximal elements $g^{\ast}$ and $T^{-1}$ in $A[T^{-1}]$,
we get $$ \gs_{T} =(\gs_1)_{T^{-1}}\circ (\gs_2)_{g^{\ast}}$$  where  
$\gs_1 \in {EO}(\mathbb{M}(P[T^{-1}]_{g^{\ast}}))$ and $\gs_2\in {EO}(\mathbb{M}(P[T,T^{-1}))$.
Now $(H(1)\gs_1)_{T^{-1}} = ((0,0,1)\gs_2^{-1})_{g^*}$. 
Patch $H(1)\gs_1$ and $(0,0,1)\gs_2^{-1}$ to obtain $w\in \mathbb{Q'}(P[T^{-1}])$ such that  
$w_{T^{-1}}=(0,0,1)\gs_2^{-1}$ in $A[T,T^{-1}]$ and $w_{g^*}= H(1)\sigma_1$ in $A[T^{-1}]_{g^*}$. In particular
$w_{T^{-1}}\equiv (0,0,1)$ mod ${EO}(\mathbb{M}(P[T,T^{-1}]))$.
By first case, we get
$w \equiv (0,0,1)$ mod ${EO}(\mathbb{M}(P[T^{-1}]))$.

As $g^{\ast}(T^{-1}=0)=1$ and $w_{g^{\ast}}=H(1)\gs_1$, we get
$w(T^{-1}=0)=H(1)\gs_1(T^{-1}=0)\equiv (0,0,1)$ mod ${EO}(\mathbb{M}(P))$ i.e. $H(1) \equiv (0,0,1)$ mod $EO(\mathbb M(P))$.
Thus $H(T) \equiv H(1) \equiv (0,0,1)$ mod ${E O}(\mathbb M(P[T]))$.
This completes the proof.
\qed

%%%%%%%%%%%%%%%%%%
\section{Monic inversion and lifting of surjections}

Let $(I,\omega_I) \in \mathcal{LO}(P) $ i.e. $\omega_I : P\surj I/I^2$ is a surjection. If $f:P\to I$ is a lift of $\omega_I$, then there exist
$s \in I$ such that $I=(f(P),s)$ with $s(1-s) \subset f(P).$ If $s(1-s) = f(p)$ for $p \in P$, then $(p,f,s) \in \mathbb{Q} (P)$. The map
$ \chi: \mathcal{LO}(P) \rightarrow \pi_0 (\mathbb{Q} (P))$  defined by  $\chi(I,\omega_I) = [(p,f,s)] \in \pi_0 (\mathbb{Q} (P)) $ is well defined
\cite[Theorem 2.7]{MM}.
Compositing $\chi$ with the isomorphism in (\ref{4}), we get a map
 $\zeta :\mathcal {LO}(P) \rightarrow \pi_0 ({\mathbb{Q'} } (P)).$

The next result is proved in Mandal-Mishra \cite[Theorem 4.3]{MM} when $k$ is infinite perfect. This condition on $k$ was stated to use \cite[Theorem 4.13]{BK}. Now we can use (\ref{4.13}). 

\begin{theorem}\label{mondal}
Let $A$ be  a regular domain of dimension $d$ essentially of finite type over an infinite field $k$ of characteristic  $\neq 2$. Let $P$ be a projective $A$-module of rank $n$ with $2n \geq d+3.$  Let $I$ be an ideal of $A$ of height $\geq n$ and $(I,\omega_I) \in \mathcal {LO}(P)$. Then $\omega_I$ lifts to a surjection $P \surj I$ if and only if $\zeta (I,\omega_I)= [(0,0,1)]  \in \pi_0({\mathbb{Q'} } (P))$.
\end{theorem}

\begin{theorem} {\label{polynomial}}
Let $A$ be  a regular domain of dimension $d$ essentially of finite type over an infinite field $k$ of characteristic  $\neq 2$. Let $P$ be a projective $A$-module of rank $n$ with $2n \geq d+3.$  Let $I$ be an ideal of $A[T]$ of height $\geq n$ and $(I,\omega_I) \in \mathcal {LO}(P[T])$. Then $\omega_I$ lifts to a surjection $P[T] \surj I$ if and only if $\zeta (I,\omega_I)= [(0,0,1)]  \in \pi_0({\mathbb{Q'} } (P[T]))$.
\end{theorem}

\proof
If $\zeta (I,\omega_I)= [(0,0,1)]  \in \pi_0({\mathbb{Q'} } (P[T]))$, then $\zeta (I(0),\omega_{I(0)})= [(0,0,1)]  \in \pi_0({\mathbb{Q'} } (P))$ .
By (\ref{mondal}),  the surjection $\omega_{I(0)} : P \surj I(0)/I(0)^2$ lifts to a surjection $\psi : P \surj I(0) $ such that $\psi \otimes A/I(0) =  \omega_{I(0)}.$ Thus $\omega_{I}$ lifts to a surjection $\Phi: P[T] \surj I/I^2T$ \cite[Remark 3.9]{BS2}. By  \cite[Theorem 4.13]{BK}, $\Phi$ has a  surjective lift $P[T] \surj I$ which is a lift of   $\omega_{I}.$

For converse,  consider the following commutative diagram
$$
\xymatrix{
         &\mathcal{LO}  (P[T])  \ar^{\zeta} [r] \ar^{T=0} [d]  & \pi_0({\mathbb{Q'} } (P[T]))  \ar^{T=0} [d]
         &
           &   \\
        & \mathcal{LO}  (P)  \ar^{\zeta } [r] &  \pi_0({\mathbb{Q'} } (P) )&
          & 
        }$$
If  $\omega_I$ lifts to a surjection $P[T] \surj I$,  then $\omega_{I(0)}$ also lifts to a surjection $P \surj I(0)$.
By (\ref{mondal}),  $\zeta (I(0),\omega_{I(0)})= [(0,0,1)]  \in \pi_0({\mathbb{Q'} } (P))$.
Since the right vertical map is bijective,  we conclude that  $\zeta  (I,\omega_I)= [(0,0,1)]  \in \pi_0({\mathbb{Q'} } (P[T]))$.
\qed

Now we will prove our main result.

\begin{theorem}\label{mt}
Let $A$ be  a regular domain of dimension $d$ essentially of finite type over an infinite field $k$ of characteristic  $\neq 2$. Let $P$ be a projective $A$-module of rank $n$ with $2n \geq d+3$.  Let $I$ be an ideal of $A[T]$ of height $n$ and
$\phi: P[T] \surj I/I^2$ be a surjection. Assume $\phi\otimes A(T)$ lifts to a surjection $ \theta: P[T]\otimes A(T) \surj IA(T)$.  Then $\phi$ lifts to a
surjection $\Phi : P[T] \surj I$.
 \end{theorem}

\proof 
Given $(I,\phi) \in \mathcal{LO}(P[T]).$ Let $\zeta (I,\phi)= [H(T)] \in \pi_0({\mathbb{Q'} } (P[T])).$  
Since $\phi \otimes A(T)$ has a surjective lift $\theta :P[T]\otimes A(T) \surj I A(T)$ and dim $A(T)=d$, by (\ref{mondal}), $\zeta(IA(T),\phi\otimes A(T))=[(0,0,1)]$ in $\pi_0(\mathbb Q'(P[T]\otimes A(T))$.
Thus there exist a monic polynomial $g  \in A[T]$ such that  $\zeta (I_g,\phi_g)= [H(T)_g] = [(0,0,1)] \in \pi_0( {\mathbb{Q'} } (P[T]_g)).$
By (\ref{monicquadratic}), $[H(T)]=[(0,0,1)]$ in 
$\pi_0( {\mathbb{Q'} } (P[T])).$ By  (\ref{polynomial}), $\phi$ lifts to a surjection $\Phi:P[T]\surj I$.
\qed

\begin{corollary}
Let $A$ be  a regular domain of dimension $d$ essentially of finite type over an infinite field $k$ of characteristic  $\neq 2$. Let $P=Q\oplus A$ be a projective $A$-module of rank $n$ with $2n \geq d+3$.  Let $I$ be an ideal of $A[T]$ of height $n$ containing a monic polynomial $f$.  Then any surjection
$\phi: P[T] \surj I/I^2$ lifts to a surjection $\Phi:P[T]\surj I$. 
\end{corollary}

\proof
Since $I_f=A[T]_f$, $I_f/I^2_f=0$, thus the surjection $pr_2:P[T]_f\surj I_f$ is a lift of $\phi_f$. By (\ref{mt}), $\phi$ lifts to the required surjection.
\qed

\medskip

\noindent
\textbf{Acknowledgement}

We would like to thank Arvind Asok, A. Stavrova and A.A. Ambily for replying to our queries.

\end{document}